\documentclass[a4paper,12pt,reqno]{amsart}

\newcommand\version{September 19, 2025}

\usepackage{amsmath}
\usepackage{amssymb}
\usepackage{amsfonts}
\usepackage{graphicx}
\usepackage{mathtools} 
\usepackage{bbm} 
\usepackage[colorlinks]{hyperref}

\renewcommand\eqref[1]{(\ref{#1})} 

\newcommand{\1}{\mathbbm{1}}
\newcommand{\R}{\mathbb R}

\newcommand{\Hei}{\mathbb H}
\newcommand{\Sph}{\mathbb{S}}

\DeclareMathOperator{\dist}{dist}

\graphicspath{ {images/} }
\title[Eigenvalue lower bounds]{Eigenvalue lower bounds\\ through a generalized inradius}

\author[R. L. Frank]{Rupert L. Frank}
\address{
	Rupert L. Frank:
	\endgraf
	Mathematisches Institut
	\endgraf
	Ludwig-Maximilians Universit\"at M\"unchen
	\endgraf
	Theresienstr. 39, 80333 M\"unchen, Germany
	\endgraf
	and 
	\endgraf
	Munich Center for Quantum Science and Technology
	\endgraf
	Schellingstr. 4, 80799 M\"unchen, Germany
	\endgraf
	and 
		\endgraf
	Mathematics 253-37, Caltech 
	\endgraf
	Pasadena,
CA 91125, USA
	\endgraf
	{\it E-mail address} {\rm   r.frank@lmu.de}
}

\author[A. Laptev]{Ari Laptev}
\address{
	Ari Laptev:
	\endgraf
	Department of Mathematics
	\endgraf
	Imperial College London
	\endgraf
	London SW7 2BX
	\endgraf
	UK
	\endgraf
	{\it E-mail address} {\rm  a.laptev@imperial.ac.uk}
}

\author[D. Suragan]{Durvudkhan Suragan}
\address{
	Durvudkhan Suragan:
	\endgraf
	Department of Mathematics
	\endgraf
 Nazarbayev University
	\endgraf
	53 Kabanbay Batyr Ave, Astana 010000
	\endgraf
	Kazakhstan
	\endgraf
	{\it E-mail address} {\rm durvudkhan.suragan@nu.edu.kz}}

\subjclass[2010]{35P15, 	58J50.}
\keywords{Robin Laplacian, Polyharmonic operator, Heisenberg group, Lowest eigenvalue, Hardy inequality}

\thanks{Date: \version \\
	The authors were supported in parts by US National Science Foundation grant DMS-1954995 (R.L.F.) and the German Research Foundation grants EXC-2111-390814868 and TRR 352-Project-ID 470903074 (R.L.F.).\\
No new data was collected or generated during the course of this research.}

\newtheoremstyle{theorem}
{10pt}          
{10pt}  
{\sl}  
{\parindent}     
{\bf}  
{. }    
{ }    
{}     
\theoremstyle{theorem}

\numberwithin{equation}{section}
\theoremstyle{plain}
\newtheorem{thm}{Theorem}[section]
\newtheorem{prop}[thm]{Proposition}

\newtheorem{lemma}[thm]{Lemma}
\newtheorem{proposition}[thm]{Proposition}

\theoremstyle{definition}

\newtheorem{rem}[thm]{Remark}

\newtheoremstyle{defi}
{10pt}          
{10pt}  
{\rm}  
{\parindent}     
{\bf}  
{. }    
{ }    
{}     
\theoremstyle{defi}

\newtheorem{remark}[thm]{Remark}
\newtheorem{remarks}[thm]{Remarks}

\begin{document}

\begin{abstract}
	Lieb has shown a lower bound on the smallest Dirichlet eigenvalue of the Laplace operator in terms of a generalized inradius. We derive similar bounds for Robin eigenvalues, for eigenvalues of the polyharmonic operator and the sub-Laplacian on the Heisenberg group. We propose a method based on Hardy inequalities that is different from Lieb's approach.
\end{abstract}

\maketitle
	
\section{Introduction}

\subsection{Aims and scope}

We are interested in lower bounds on the first eigenvalue of Laplace-type operators on domains in terms of simple geometric characteristics of these domains. One such characteristic is the volume of this domain and, indeed, in many situations one can prove such a bound. Instead, here we are interested in a characteristic that is more sensitive to the geometry of the domain than the volume, namely a generalized inradius. 

In the model case of the Dirichlet Laplacian on an open set of Euclidean space such bounds are well known and will be recalled momentarily. In the present paper we prove corresponding bounds for the following operators
\begin{itemize}
	\item the Robin Laplacian with a positive Robin function,
	\item the polyharmonic operators with Dirichlet boundary conditions,
	\item the Heisenberg Laplacian with Dirichlet boundary conditions.
\end{itemize}
We emphasize that these are only a few selected instances where the methods we use are applicable and that many more examples could be added.

\subsection{Reminder on lower bounds for the Dirichlet Laplacian}

To set the scene, we recall some classical lower bounds on the first eigenvalue of the Dirichlet Laplacian on an open set $\Omega\subset\R^d$,
\begin{equation}\label{Laplacian-Dirichlet}
	\begin{cases}
		\begin{aligned}
			& -\Delta u = \lambda^{\rm D}_\Omega u  & & \text{in}\ \Omega \,, \\
			& u =0 & & \text{on}\ \partial\Omega \,.
		\end{aligned}
	\end{cases}
\end{equation}
In his famous book The Theory of Sound (first edition in 1877), Rayleigh claimed that the disk should be the planar domain that minimizes the first (lowest) eigenvalue of the Laplacian with Dirichlet boundary conditions among all domains of the same area. The proof of this conjecture, in fact of its extension to any dimension, came almost 30 years later and is due to G.~Faber and E.~Krahn. The Rayleigh--Faber--Krahn theorem asserts: For any open set $\Omega\subset\R^d$ of finite measure, the first (lowest) eigenvalue of the Dirichlet Laplacian on $\Omega\subset\R^d$ satisfies
\begin{equation}\label{Rayleigh-Faber-Krahn}
	\lambda^{\rm D}_\Omega \geq\lambda^{\rm D}_{\Omega^*} \,,
\end{equation}
where $\Omega^*\subset\R^d$ is the ball with $|\Omega^*|=|\Omega|$. The proof of \eqref{Rayleigh-Faber-Krahn} proceeds by symmetrization. By scaling, we can write the inequality as
$$
\lambda^{\rm D}_\Omega \geq C_d \, |\Omega|^{-2/d}
$$
with $C_d= \lambda^{\rm D}_B |B|^{2/d}$, where $B\subset\R^d$ is the unit ball. This is a (sharp) lower bound on $\lambda^{\rm D}_\Omega$ in terms of the volume of $\Omega$.

We now turn to lower bounds of a different type. For a restricted class of open sets $\Omega$, one can bound $\lambda^{\rm D}_\Omega$ from below in terms of the inradius
$$ 
R_{\Omega}=\sup _{x \in \Omega} \min _{y \in \partial \Omega}|x-y| \,,
$$
that is, the radius of the largest ball contained in $\Omega$. In the formulation of these results, $\lambda_\Omega^{\rm D}$ denotes the infimum of the spectrum of the selfadjoint realization of $-\Delta$ in $L^2(\Omega)$ with Dirichlet boundary conditions. This number may or may not be an eigenvalue.

Payne and Stakgold \cite{PaSt} and Li and Yau \cite{LiYa} showed that when $\Omega$ is mean-convex, then
\begin{equation}
	\label{eq:hersch}
	\lambda^{\rm D}_\Omega \geq \frac{\pi^2}{4} \, R_\Omega^{-2} \,.
\end{equation}
For convex sets, this had earlier been shown by Hersch \cite{Her0} and Protter \cite{Pr}; see also \cite[Proposition 2.55]{FrLaWe}.

In two dimensions, it can also be shown that $R_\Omega^2 \lambda_\Omega^{\rm D}$ is bounded from below by a positive constant among all simply-connected domains $\Omega$. This is due to Makai \cite{Ma}; see also \cite{Ash}. For general open sets $\Omega\subset\R^d$, however, it is easy to see that there is no positive lower bound on $R_\Omega^2 \lambda_\Omega^{\rm D}$ that depends only on the dimension $d\geq 2$.

In 1983, Lieb \cite{Lieb83} suggested to use instead of the inradius $R_\Omega$, which corresponds to the largest ball contained in $\Omega$, a certain \emph{generalized inradius} that corresponds to a ball that intersects $\Omega$ `significantly'. He introduced the generalized inradius
$$
R_\Omega^{(\psi)} := \sup\left\{ r >0 :\ \sup_{x\in\Omega} \frac{|\Omega\cap B_r(x)|}{|B_r(x)|} \geq \psi \right\},
$$
where the parameter $\psi\in(0,1)$ quantifies how significantly a ball can intersect $\Omega$, and showed the lower bound
$$
\lambda^{\rm D}_\Omega \geq \lambda^{\rm D}_B \, (\psi^{-2/d}-1) \, ( R_\Omega^{(\psi)})^{-2} \,.
$$
Here $B\subset\R^d$ denotes, as before, the unit ball. Avoiding the notion of the generalized inradius, we can deduce from Lieb's theorem the lower bound
\begin{equation}
	\label{eq:liebbound}
	\lambda^{\rm D}_\Omega \geq \frac{\lambda^{\rm D}_B}{r^2} \left( \left( \sup_{x\in\Omega} \frac{|\Omega\cap B_r(x)|}{|B_r(x)|} \right)^{-2/d} - 1 \right),
\end{equation}
which is valid for all $r>0$. It is bounds of the latter form that we will be concerned with in the remainder of this paper.

In passing, we also mention work by Maz'ya and Shubin \cite{MaSh}, where a different and finer notion of generalized inradius is used. Specifically, the significance of the intersection of $\Omega$ with a ball is measured in terms of capacity rather than measure.

In his textbook \cite{Dav89}, Davies proved a bound that is worse than Lieb's bound \eqref{eq:liebbound}, but shares the same quantitative features, namely, the lower bound on $r^2 \lambda_\Omega^{\rm D}$ diverges like $\Psi_r^{-2/d}$ as
\begin{equation}
	\label{eq:defpsir}
	\Psi_r := \sup_{x\in\Omega} \frac{|\Omega\cap B_r(x)|}{|B_r(x)|}
\end{equation}
tends to zero and vanishes like $1-\Psi_r$ as $\Psi_r$ tends to one. Related bounds appear in \cite{CrDe} in the setting of manifolds. Unaware of the argument in \cite{Dav89}, the authors of \cite{FL21} found a similar argument that yields a very explicit lower bound on $r^2 \lambda_\Omega^{\rm D}$ that vanishes linearly as $\Psi_r\to 1$. Among other things, in this paper, we further improve this technique and derive a bound that is precisely of Lieb's form \eqref{eq:liebbound}, except for a different, dimension-dependent constant.

More important than this result, however, is the fact that in contrast to Lieb's proof ours does not rely on the Rayleigh--Faber--Krahn theorem. This opens up the possibility of deriving similar bounds for operators for which symmetrization results are not known. Examples include the polyharmonic operator and the Laplacian on the Heisenberg group, each one considered with Dirichlet boundary conditions. Even if for the Laplacian with a positive Robin function a Rayleigh--Faber--Krahn theorem is valid, we do not see how to adapt Lieb's proof to this case, while, as we will show, the approach that we are proposing does work.


\subsection{Our three examples. Previous results}\label{sec:models}

Let us give an informal presentation of the three model problems that we are considering and recall some known results.

\subsubsection*{Robin Laplacian}

For an open set $\Omega\subset\R^d$ with sufficiently regular boundary and a positive constant $\sigma$, the Robin eigenvalue problem is
\begin{equation}\label{Laplacian-Robin}
	\begin{cases}
		\begin{aligned}
		& -\Delta u = \lambda^{{\rm R}\sigma}_\Omega u  & & \text{in}\ \Omega \,, \\
		& \frac{\partial u}{\partial\nu}+\sigma u =0 & & \text{on}\ \partial\Omega \,.
		\end{aligned}
	\end{cases}
\end{equation}
Here $\nu$ denotes the outward unit normal on $\partial\Omega$ and $\lambda^{{\rm R}\sigma}_\Omega$ is the smallest eigenvalue. The positivity of $\sigma$ implies that $\lambda^{{\rm R}\sigma}$ is nonnegative (and typically positive).

An analogue of the Rayleigh--Faber--Krahn theorem is valid, as shown in \cite{Dan06} and the references cited therein. An analogue for the Hersch--Protter bound \eqref{eq:hersch} is valid when $\Omega$ is convex. Indeed, Kovařík \cite{Kov14} showed
\begin{equation}\label{Hynek}
	\lambda^{{\rm R}\sigma}_\Omega \geq \frac{1}{4} \frac{\sigma}{R_{\Omega}\left(1+\sigma R_{\Omega}\right)} \,. 
\end{equation}
In contrast to \eqref{eq:hersch}, however, it is unclear whether the constant $1/4$ is sharp. In the appendix to this paper, we extend the validity of \eqref{Hynek} to mean-convex sets and improve the constant for large values of $\sigma R_\Omega$.


\subsubsection*{Polyharmonic operator}

For an open set $\Omega\subset\R^d$ and an integer $m\in\mathbb N=\{1,2,3,\ldots\}$, we are interested in the Dirichlet problem for the $m$-harmonic operator,
\begin{equation}\label{eq:polyharmonic}
	\begin{cases}
		\begin{aligned}
			& (-\Delta)^m u = \lambda^{(m)}_\Omega u  & & \text{in}\ \Omega \,, \\
			& \partial^\alpha u =0 & & \text{on}\ \partial\Omega \ \text{for all}\ |\alpha|<m \,.
		\end{aligned}
	\end{cases}
\end{equation}

The validity of a Rayleigh--Faber--Krahn theorem, that is, the question whether balls minimize the lowest eigenvalue of the polyharmonic operator among sets with a given volume, is a major open problem. The only results that we are aware of concern the bi-harmonic case ($m=2$) in two \cite{Nadir95} and three \cite{AB95} dimensions.

An analogue of the Hersch--Protter bound \eqref{eq:hersch} for convex $\Omega$, that is a lower bound on the lowest eigenvalue in terms of the inradius, was shown by Owen in \cite{Owen99}. As far as we know, the sharp constant in this inequality is not known when $m\geq 2$.


\subsubsection*{Sub-Laplacian on the Heisenberg group}

Let $N\geq 1$. We work on the Heisenberg group $\Hei^N = \R^{2N+1}$ with the group operation defined as follows. Denoting elements of $\Hei^N$ by $(z,t)$ and $(z',t')$ with $z,z'\in\R^{2N}$, $t,t'\in\R$, we have
$$
(z,t)\circ(z',t') = (z+z',t+t' + \frac12 z\cdot J z') \,,
$$
where $J$ is the $(2N)\times(2N)$ matrix defined by
$$
J = \begin{pmatrix} 0 & 1 \\ - 1 & 0 \end{pmatrix}
$$
with $0$ and $1$ denoting the zero and identity $N\times N$ matrix, respectively. The dot $\cdot$ denotes the ordinary scalar product in $\R^{2N}$. The Lebesgue measure on $\R^{2N+1}$ is a (left- and right-)Haar measure for this group and integration will always be understood with respect to this measure.

We introduce the first order differential operators
$$
Z_n := \partial_{z_n} -\frac12 \ e_n\cdot Jz \ \partial_t \,,
\qquad
n=1,\ldots,2N \,.
$$
Here $e_n$ is the $n$-th standard unit vector in $\R^{2N}$.

The Dirichlet eigenvalue problem on an open set $\Omega\subset\Hei^N$ is
\begin{equation}\label{Laplacian-Dirichletheisen}
	\begin{cases}
		\begin{aligned}
			& -\sum_{n=1}^{2N} Z_n^2 u = \lambda^{\rm H}_\Omega u  & & \text{in}\ \Omega \,, \\
			& u =0 & & \text{on}\ \partial\Omega \,,
		\end{aligned}
	\end{cases}
\end{equation}
and again we are interested in lower bounds on $\lambda^{\rm H}_\Omega$.

The analogue of the Rayleigh--Faber--Krahn theorem is not known. The question of good (even if not sharp) constants on $|\Omega|^{1/(N+1)}\lambda_\Omega^{\rm H}$ is investigated in \cite{FrHe}. Analogues of Croke's bounds on $\lambda_\Omega^{\rm H}$ appear in \cite{PrRiSe}.


\section{Main results}

In this section we present the main results of this paper in the three problems described in the previous subsection.

\subsection{The Robin Laplacian}

Let $d\geq 2$ and let $\Omega\subset\R^d$ be an open set whose boundary is uniformly Lipschitz continuous in the sense of \cite[Definition 13.11]{Le}. Under this assumption, it is known \cite[Theorem 18.40]{Le} that there is a bounded trace operator from $H^1(\Omega)$ to $L^2(\partial\Omega)$, where the latter space is defined with respect to the usual surface measure, denoted by $dS$. For any constant $\sigma>0$ the quadratic form
$$
\int_\Omega |\nabla u(x)|^2\,dx + \sigma \int_{\partial\Omega} |u(y)|^2\,dS(y) \,,
$$
defined for $u\in H^1(\Omega)$, is closed in $L^2(\Omega)$ and therefore generates a selfadjoint operator in $L^2(\Omega)$. We let $\lambda_\Omega^{\rm R\sigma}$ denote the bottom of the spectrum of this operator.

We recall that the quantity $\Psi_r$ is defined in \eqref{eq:defpsir}.
	
\begin{thm}\label{main} 
	Let $\Omega\subset \mathbb{R}^d$ be open with uniformly Lipschitz continuous boundary and let $\sigma>0$. Then, for any $r>0$,
	$$
	\lambda_\Omega^{{\rm R}\sigma} \geq d\, \frac{\sigma^2}{(1+2\sigma r)^2} \left( 1- \Psi_r \right).
	$$
\end{thm}

\begin{rem}
	The Robin Laplacian can also be defined for a positive Robin function $\sigma$ on $\partial\Omega$, not necessarily a constant. Denoting by $\lambda^{{\rm R}\sigma}_\Omega$ the infimum of the spectrum even in this more general situation, we deduce from the inequality $\sigma\geq\mathrm{ess}\inf\sigma =:\underline\sigma$ and the variational principle that $\lambda^{{\rm R}\sigma}_\Omega \geq \lambda^{{\rm R}\underline\sigma}_\Omega$. Assuming $\underline\sigma>0$, we obtain a lower bound on $\lambda^{{\rm R}\underline\sigma}_\Omega$ by Theorem \ref{main}. Consequently, we also get a lower bound on $\lambda^{{\rm R}\sigma}_\Omega$ in terms of $\underline\sigma>0$.
\end{rem}

Theorem \ref{main} is proved in Section \ref{sec:robin}.


\subsection{The polyharmonic operator}

Let $\Omega\subset\R^d$, $d\geq 2$, be an open set and let $m\in\mathbb N=\{1,2,3,\ldots\}$. We are interested in the Dirichlet realization of the polyharmonic operator $(-\Delta)^m$ in $\Omega$. This is the selfadjoint operator in $L^2(\Omega)$ generated by the quadratic form
$$
h_\Omega^{(m)}[u] :=
\begin{cases} 
	\int_\Omega |(-\Delta)^{m/2}u(x)|^2\,dx & \text{if}\ m \ \text{is even}\,,\\
	\int_\Omega |\nabla(-\Delta)^{(m-1)/2}u(x)|^2\,dx & \text{if}\ m \ \text{is odd}\,,
\end{cases}
$$
defined for $u\in H^m_0(\Omega)$, the closure of $C^\infty_c(\Omega)$ in $H^m(\Omega)$. The number $\lambda_\Omega^{(m)}$ denotes the infimum of the spectrum of this operator.

\begin{thm}\label{mainPolyharmonic}
	Let $\Omega \subset \mathbb{R}^d$ be open and let $m\in\mathbb N$. Then, for any $r>0$,
	\begin{equation}\label{eq:mainpoly1}
		\lambda^{(m)}_{\Omega} 
		\geq C_{m,d} \, \frac{1}{r^{2m}} \left( \Psi_r^{-2m/d} - 1 \right)			
	\end{equation}
	and
	\begin{equation}
		\label{eq:mainpoly2}
		\lambda^{(m)}_{\Omega} 
		\geq c_{m,d} \, C_{m,d} \, \frac{1}{r^{2m}} \left( 1- \Psi_r \right)
	\end{equation}
	where
	\begin{equation}
		\label{eq:owenconst}
			C_{m,d} :=\frac{(d+2 m-2)(d+2 m-4) \cdots d \cdot (2 m-1) (2 m-3) \cdots 1}{4^m}
	\end{equation}
	and
	\begin{equation}
		\label{eq:polyharmconst}
		c_{m,d} := \left( \frac{d+2m}{2m} \right)^{2m/d} \frac{d+2m}{d} \,. 
	\end{equation}
\end{thm}

\begin{remarks}
	\begin{enumerate}
		\item[(a)] Let us compare the bounds \eqref{eq:mainpoly1} and \eqref{eq:mainpoly2}. In the regime where $\Psi_r$ is close to $0$, the bound \eqref{eq:mainpoly1} is clearly superior. In the opposite regime where $\Psi_r$ is close to $1$, the right sides of both bounds vanish linearly with $1-\Psi_r$, but the coefficient in \eqref{eq:mainpoly2} is superior. For this reason, none of the bounds implies the other.
		
		\medspace
		
		\item[(b)] In the case $m=1$, the bounds \eqref{eq:mainpoly1} and \eqref{eq:mainpoly2} read
		\begin{equation}
			\label{eq:davieslieb1}
			\lambda_\Omega^{\rm D} \geq \frac d4 \, \frac 1{r^2} \left( \Psi_r^{-2/d} - 1 \right)
		\end{equation}
		and
		\begin{equation}
			\label{eq:davieslieb2}
			\lambda_\Omega^{\rm D} \geq \left( \frac{d+2}{2} \right)^{2/d} \frac{d+2}{4} \, \frac{1}{r^2} \left( 1 - \Psi_r \right).
		\end{equation}		
		Let us compare these bounds with existing results in the literature.
		
		The bound \eqref{eq:davieslieb1} is precisely of the same form as Lieb's bound \eqref{eq:liebbound}, with the only exception that the constant $d/4$ in \eqref{eq:davieslieb1} is replaced by the constant $\Lambda_B^{\rm D}$, the first eigenvalue of the Dirichlet Laplacian on the unit ball, in \eqref{eq:liebbound}. We have $\lambda_B^{\rm D}=j_{(d-2)/2,1}^2$, where $j_{\nu,1}$ is the first zero of the Bessel function $J_\nu$ of order $\nu$. Using well-known bounds on zeros of Bessel functions, one can show that $\lambda_B^{\rm D}>d/4$, so Lieb's bound is better than \eqref{eq:davieslieb1}.
		
		In \cite{Dav89}, Davies proves a bound similar to \eqref{eq:davieslieb1}, which shares the features that it blows up like $\Psi_r^{-2/d}$ as $\Psi_r \to 0$ and that vanishes linearly as $\Psi_r \to 1$. The constants in \eqref{eq:davieslieb1}, however, are better than the constants in \cite{Dav89}.
		
		In \cite{FL21}, a bound of the form \eqref{eq:davieslieb2} is shown, but the constant in \eqref{eq:davieslieb2} is better.
		
		\medspace
		
		\item[(c)] We do not know whether the blow up like $\Psi_r^{-2/d}$ as $\Psi_r\to 0$ or the linear vanishing as $\Psi_r\to1$ in a bound like \eqref{eq:davieslieb1} is optimal.		
	\end{enumerate}
\end{remarks}

Theorem \ref{mainPolyharmonic} is proved in Section \ref{sec:polyharm}.


\subsection{The Heisenberg Laplacian}

We shall use the notation from Subsection \ref{sec:models}. Sometimes we write $p$, rather than $(z,t)$, for points in $\Hei^N$. 

Let $\Omega\subset\Hei^N$ be an open set. The Folland--Stein Sobolev space $S^1_0(\Omega)$ is defined as the closure of $C^1_c(\Omega)$ with respect to the quadratic form
$$
\int_\Omega \sum_{n=1}^{2N} |(Z_n u)(p)|^2\,dp + \int_\Omega |u(p)|^2\,dp \,.
$$
The Dirichlet sub-Laplacian on $\Omega$ is the selfadjoint operator in $L^2(\Omega)$ generated by the quadratic form
$$
\int_\Omega \sum_{n=1}^{2N} |(Z_n u)(p)|^2\,dp \,,
$$
defined for $u\in S^1_0(\Omega)$. The number $\lambda_\Omega^{\rm H}$ denotes the infimum of the spectrum of this operator. 

To formulate our lower bound on $\lambda_\Omega^{\rm H}$ we need to introduce some notation. For $p\in\Hei^N$ we set
$$
H(p) := \{ p\circ (z',0) :\ z'\in\R^{2N} \} \,.
$$
This is the affine hyperplane through the point $p=(z,t)$ with normal vector $(\frac12 Jz,1)^{\rm T}$. The induced Lebesgue measure on $H(p)$ is denoted by $\sigma$. (We hope that this does not create any confusion with the use of this letter in the context of the Robin problem.) Moreover, for $r>0$ we set
$$
B_r^{\rm H}(p) := \{ p\circ (z',0):\ |z'|< r \} \,,
$$
where $|z'|$ denotes the Euclidean length of $z'\in\R^{2N}$. Finally, let
$$
\tilde\Psi_r := \sup_{p\in\Omega} \frac{\sigma(\Omega\cap B^{\rm H}_r(p))}{\sigma(B_r^{\rm H}(p))}
\qquad\text{for all}\ r>0 \,.
$$

\begin{thm}\label{mainheisen}
	Let $\Omega\subset\Hei^N$ be open. Then, for any $r>0$,
	$$
	\lambda_\Omega^{\rm H} \geq \frac N2 \, \frac1{r^2} \left( \tilde\Psi_r^{-1/N} -1 \right)
	$$
	and	
	$$
	\lambda_\Omega^{\rm H} \geq \frac{(N+1)^{(N+1)/N}}{2} \, \frac1{r^2} \left( 1 - \tilde\Psi_r \right).
	$$
\end{thm}

This theorem says, in particular, that if $r^2 \lambda_\Omega^{\rm H}$ is small, then one of the hyperplane sections $\Omega\cap H(p)$, $p\in\Omega$, contains a large fraction of a ball with radius $r$ in this hyperplane.

Theorem \ref{mainheisen} is proved in Section \ref{sec:heisen}.


\section{The technical lemma}

The strategy for proving our main results (Theorems \ref{main}, \ref{mainPolyharmonic} and \ref{mainheisen}) is based on Hardy inequalities of a form reminiscent of the one by Davies \cite{Da84}. In order to convert these Hardy inequalities into eigenvalue bounds we need pointwise lower bounds on the corresponding Hardy weights. In the Euclidean setting of Theorems \ref{main} and \ref{mainPolyharmonic} such pointwise lower bounds are the topic of the present section. A similar strategy is used in the Heisenberg setting in Section \ref{sec:heisen}. The method that we are using is a refinement of that in \cite{FL21}.

Let $\Omega\subset\R^d$ be an open set and set
\begin{align}
	\label{eq:defdeltaomega}
	& \delta_\omega(x) := \inf\{ |s| :\ x+s\omega\not\in\Omega \} & & \text{for all}\ x\in\Omega \,,\ \omega\in\Sph^{d-1} \,,\\
	\label{eq:defpsirx}
	& \psi_r(x) := \frac{|\Omega\cap B_r(x)|}{|B_r(x)|} & & \text{for all}\ x\in\Omega \,,\ r>0 \,.
\end{align}

\begin{lemma}\label{lemma}
	Let $\Omega\subset\R^d$ be open and $\alpha>0$. Then, for any $x\in\Omega$ and any $r>0$,
	\begin{equation}
		\label{eq:lemma1}
		|\Sph^{d-1}|^{-1} \int_{\Sph^{d-1}} \frac{d\omega}{\delta_\omega(x)^\alpha}
		\geq \frac{1}{r^\alpha} \left( \left( \psi_r(x) \right)^{-\alpha/d} - 1 \right)
	\end{equation}
	and
	\begin{equation}
		\label{eq:lemma2}
		|\Sph^{d-1}|^{-1} \int_{\Sph^{d-1}} \frac{d\omega}{\delta_\omega(x)^\alpha}
		\geq \frac{1}{r^\alpha} \left( \frac{d+\alpha}{\alpha} \right)^{\alpha/d} \frac{d+\alpha}{d} \left( 1 - \psi_r(x) \right).
	\end{equation}
	Moreover, for all $\ell\geq 0$,
	\begin{equation}
		\label{eq:lemma3}
		|\Sph^{d-1}|^{-1} \int_{\Sph^{d-1}} \frac{d\omega}{( \delta_\omega(x) + \ell)^{\alpha}} 
		\geq \frac{1}{(r+\ell)^\alpha} \left( 1 - \psi_r(x) \right).
	\end{equation}
\end{lemma}

\begin{proof}
	\emph{Step 1.} We begin by proving \eqref{eq:lemma1}.
	
	By the layer cake formula, we can write
	$$
	r^\alpha |\Sph^{d-1}|^{-1} \int_{\Sph^{d-1}} \frac{d\omega}{\delta_\omega(x)^\alpha} = \int_0^\infty s(\tau^{1/\alpha})\,d\tau = \alpha \int_0^\infty s(t)\, t^{\alpha-1}\,dt
	$$
	with
	$$
	s(t) := \frac{|\{ \omega\in\Sph^{d-1}:\ r/ \delta_\omega(x) \geq t \}|}{|\Sph^{d-1}|} \,.
	$$
	(We consider $x$ and $r$ as fixed and do not reflect their dependence in our notation.)
	Meanwhile, we have, since $\1(x+\rho\omega\in\Omega)\geq\1(\delta_\omega(x)>\rho)$,
	\begin{align*}
		\psi_r(x) & = \frac{\int_{\Sph^{d-1}} \int_0^r \1(x+\rho\omega\in\Omega) \rho^{d-1}\,d\rho\,d\omega}{\int_{\Sph^{d-1}} \int_0^r \rho^{d-1}\,d\rho\,d\omega} \\
		& \geq \frac{\int_{\Sph^{d-1}} \int_0^r \1(\delta_\omega(x)>\rho) \rho^{d-1}\,d\rho\,d\omega}{\int_{\Sph^{d-1}} \int_0^r \rho^{d-1}\,d\rho\,d\omega} \\
		& = |\Sph^{d-1}|^{-1} \int_{\Sph^{d-1}} \min\{1 , (\delta_\omega(x)/r)^d\}\,d\omega \\
		& = d \int_0^1 \frac{|\{\omega\in\Sph^{d-1}:\ \delta_\omega(x)/r > \sigma \}|}{|\Sph^{d-1}|} \, \sigma^{d-1}\,d\sigma \\
		& = d \int_0^1 (1-s(1/\sigma)) \,\sigma^{d-1}\,d\sigma \\
		& = 1 - d \int_0^1 s(1/\sigma ) \,\sigma^{d-1}\,d\sigma \\ 
		& = 1 - d \int_1^\infty s(t) \, t^{-d-1}\,dt \,.
	\end{align*}
	Here we used again the layer cake formula, together with a change of variables.
	
	As a consequence of these expressions, we see that the claimed inequality will follow, once we have shown
	\begin{equation}
		\label{eq:functionineqality}
			\alpha \int_0^\infty s(t)\,t^{\alpha-1}\,dt \geq \left( 1 - d \int_1^\infty s(t) \, t^{-d-1}\,dt \right)^{-\alpha/d} - 1 \,.
	\end{equation}
	We will show that this inequality holds for any function $s$ on $\R_+$ satisfying $0\leq s\leq 1$. Since our function $s$ satisfies these properties, we will have verified \eqref{eq:lemma1}. (Note also that our function $s$ is nonincreasing, but we do not use this property.)
	
	We bound the left side from below by
	$$
	\alpha \int_0^\infty s(t)\,t^{\alpha-1} \,dt \geq \alpha \int_1^\infty s(t)\,t^{\alpha-1}\,dt =: M \,.
	$$
	Meanwhile, by the bathtub principle, since $t^{-d-\alpha}$ is decreasing on $(1,\infty)$, the supremum of $\int_1^\infty s(t)\, t^{-d-1}\,dt$ over all functions $s$ on $(1,\infty)$ with $0\leq s\leq 1$ and integral (against $\alpha t^{\alpha-1}dt$) equal to $M$ is attained at $\1(1<\tau<(M+1)^{1/\alpha})$. Thus,
	$$
	d \int_1^\infty s(t) \,t^{-d-1}\,dt \leq d \int_1^{(M+1)^{1/\alpha}} t^{-d-1}\,dt = 1-(M+1)^{-d/\alpha}
	$$
	and the right side above is bounded from above by
	$$
	\left( 1 - d \int_1^\infty s(t) \, t^{-d-1}\,dt \right)^{-\alpha/d} - 1
	\leq \left( 1 - (1-(M+1)^{-d/\alpha} ) \right)^{-\alpha/d} - 1 = M \,.
	$$
	This proves the claimed inequality.
	
	\medspace
	
	\emph{Step 2.} We prove the bounds \eqref{eq:lemma2} and \eqref{eq:lemma3}.
	
	We could pass to an inequality for the function $s$ as in Step 1, but it is slightly easier to work directly with the function $\delta_\omega$. (The two approaches lead to the same constant, as one can check.) The proof in Step 1 shows that
	\begin{align*}
		1 - \psi_r(x) & \leq |\Sph^{d-1}|^{-1} \int_{\Sph^{d-1}} \left( 1- \min\{1,(\delta_\omega(x)/r)^d\}\right)d\omega \\
		& = |\Sph^{d-1}|^{-1} \int_{\Sph^{d-1}} \left( 1- (\delta_\omega(x)/r)^d\right)_+ d\omega \,.
	\end{align*}
	We now apply the elementary inequality
	\begin{equation}
		\label{eq:elementary}
			(1-X)_+ \leq \frac{\beta^\beta}{(\beta+1)^{\beta+1}} \ X^{-\beta}
			\qquad\text{for all}\ X>0 \,,\ \beta>0 \,,
	\end{equation}
	with $X= (\delta_\omega(x)/r)^d$ and $\beta=\alpha/d$ and obtain the claimed bound \eqref{eq:lemma2}. 
	
	To prove \eqref{eq:lemma3} we combine the elementary inequality \eqref{eq:elementary} with the trivial bound $(1-X)_+\leq 1$ and obtain
	\begin{align*}
		1 - \psi_r(x) & \leq |\Sph^{d-1}|^{-1} \int_{\Sph^{d-1}} \min\left\{1, \left( \frac{\alpha}{\alpha+d} \right)^{\alpha/d} \frac{d}{\alpha+d} \left( \frac{r}{\delta_\omega(x)} \right)^\alpha \right\} d\omega \\
		& \leq |\Sph^{d-1}|^{-1} \int_{\Sph^{d-1}} \min\left\{1, \left( \frac{r}{\delta_\omega(x)} \right)^\alpha \right\} d\omega \\
		& \leq |\Sph^{d-1}|^{-1} \int_{\Sph^{d-1}} \left( \frac{r+\ell}{\delta_\omega(x)+\ell} \right)^\alpha d\omega \,.
	\end{align*}
	This concludes the proof of \eqref{eq:lemma3}.
\end{proof}



\section{Proof of lower bounds for the Robin Laplacian}\label{sec:robin}

Our goal in this section is to prove Theorem \ref{main}. We recall that for every open set $\Omega\subset\R^d$ we defined $\delta_\omega$ in \eqref{eq:defdeltaomega}. For $\sigma>0$ we set
$$
\mu_\sigma(x):=d\left|\mathbb{S}^{d-1}\right|^{-1} \int_{\mathbb{S}^{d-1}}\left(\delta_\omega(x)+
\frac{1}{2 \sigma}\right)^{-2} d \omega
\qquad\text{for all}\ x\in\Omega \,.
$$
The following is a special case of the result in \cite{KL12}, where even the case of a non-constant Robin function is treated.

\begin{prop}\label{Hardyinequality}\cite{KL12}
	Let $\Omega\subset\R^d$ be open with uniformly Lipschitz continuous boundary and let $\sigma>0$. Then, for all $u\in H^1(\Omega)$,
	$$
	\int_{\Omega}|\nabla u(x)|^2 \,d x+\sigma \int_{\partial \Omega} |u(y)|^2 \,d S(y) \geq \frac{1}{4} 
	\int_{\Omega}|u(x)|^2 \mu_\sigma(x) \,d x .
	$$
\end{prop}

\begin{proof}[Proof of Theorem \ref{main}] 
	According to Lemma \ref{lemma} with $\alpha=2$ and $\ell=(2\sigma)^{-1}$, for any $r>0$ we have the pointwise bound
	$$
	\mu_\sigma(x) \geq d \left( r+ \frac1{2\sigma} \right)^{-2} \left( 1-\psi_r(x) \right).
	$$
	Taking the infimum over $x\in\Omega$ gives
	$$
	\mu_\sigma(x) \geq d \left( r+ \frac1{2\sigma} \right)^{-2} \left( 1-\Psi_r \right).
	$$
	Inserting this bound into the right side of the inequality in Proposition \ref{Hardyinequality} gives the claimed lower bound on $\lambda^{{\rm R} \sigma}_\Omega$.
\end{proof}


\section{Proof of lower bounds for the polyharmonic operator}\label{sec:polyharm}

Our goal in this section is to prove Theorem \ref{mainPolyharmonic}. Our starting point is once more a Hardy inequality, this time for the polyharmonic operator due to Owen \cite{Owen99}. It generalizes Davies's Hardy inequality \cite{Da84}. We set
$$
M_\Omega^{(m)}(x) := \frac{1}{\left|\mathbb{S}^{d-1}\right|} \int_{\mathbb{S}^{d-1}} \frac{d\omega}{\delta_\omega(x)^{2 m}}
\qquad\text{for all}\ x\in\Omega
$$
and recall the definition of the constant $C_{m,d}$ from \eqref{eq:owenconst}.

\begin{proposition}\label{hardyowen}\cite{Owen99}
	Let $\Omega\subset\R^d$ be open and let $m\in\mathbb N$. Then, for all $u\in H^m_0(\Omega)$,
	$$
	h_\Omega^{(m)}[u] \geq C_{m,d} \int_\Omega |u(x)|^2 M_\Omega^{(m)}(x)\,dx \,.
	$$
\end{proposition}

\begin{proof}[Proof of Theorem \ref{mainPolyharmonic}]
	According to Lemma \ref{lemma} with $\alpha=2m$, for any $r>0$ we have the pointwise bounds
	$$
	M_\Omega^{(m)}(x) \geq \frac1{r^{2m}} \left( (\psi_r(x))^{-\frac{2m}d}-1 \right)
	$$
	and, with $c_{d,m}$ defined in \eqref{eq:polyharmconst},
	$$
	M_\Omega^{(m)}(x) \geq c_{d,m} \frac1{r^{2m}} \left( 1- \psi_r(x) \right).
	$$
	Taking the infimum over $x\in\Omega$ gives
	$$
	M_\Omega^{(m)}(x) \geq \frac1{r^{2m}} \left( \Psi_r^{-\frac{2m}d}-1 \right)
	$$
	and
	$$
	M_\Omega^{(m)}(x) \geq c_{d,m} \frac1{r^{2m}} \left( 1 - \Psi_r \right).
	$$
	Inserting these bounds into the right side of the inequality in Proposition \ref{hardyowen} gives the claimed lower bound on $\lambda_\Omega^{(m)}$.
\end{proof}


\section{Proof of lower bounds on the Heisenberg group}\label{sec:heisen}

Our goal in this section is to prove Theorem \ref{mainheisen}. We proceed again by combining a Hardy inequality with an elementary lemma similar to \eqref{lemma}.

To state the relevant Hardy inequality, we recall the notation of the Heisenberg group from Subsection \ref{sec:models}. Let $\Omega\subset\Hei^N$ be an open set. In order to define the Davies--Hardy distance, for $p\in\Omega$ and $\omega\in\Sph^{2N-1}$ we set
$$
\tilde\delta_\omega(p) := \inf\{ |s| :\ p \circ (s\omega,0) \not\in\Omega \}
$$
and
$$
\tilde\delta(p) := \left( 2N |\Sph^{2N-1}|^{-1} \int_{\Sph^{2N-1}} \tilde\delta_\omega(p)^{-2}\,d\omega \right)^{-1/2} \,.
$$
The dependence on the set $\Omega$ is not reflected in these notations.

\begin{prop}\label{hardyheisenberg}
	Let $\Omega\subset\Hei^N$ be an open set. Then for all $u\in S^1_0(\Omega)$
	$$
	\int_\Omega \sum_{n=1}^{2N} |(Z_n u)(p)|^2\,dp \geq \frac14 \int_\Omega \frac{|u(p)|^2}{\tilde\delta(p)^2}\,dp \,.
	$$
\end{prop}

This inequality essentially appears in \cite[Proposition 2]{PrRiSe}. There are two differences, however. First, the definition of $\tilde\delta_\omega(p)$ in \cite{PrRiSe} involves $s$ rather than $|s|$, so our $\tilde\delta_\omega(p)$ is not larger than theirs and so is our $\delta(p)$, making our inequality tighter. Second, in \cite{PrRiSe} it is assumed that the boundary of $\Omega$ is at least Lipschitz continuous, while we prove the inequality without any boundary regularity assumption whatsoever. The proof in \cite{PrRiSe} proceeds by deriving a sub-Riemannian analogue of the Santal\'o formula, which is where some boundary regularity enters the proof. In the special case of the Heisenberg group, this formula already appeared in \cite{Pa} and \cite{Mo}. Our proof is based on a much simpler change of variables formula, which is independent of any regularity considerations of the boundary of $\Omega$. It should be emphasized, however, that in \cite{PrRiSe} a much more general setting than the Heisenberg group is considered. It is not clear to us how far our method of proof can be generalized.

For a Hardy inequality on $\Hei^N$ of a different type we refer to \cite{La}.

\begin{remark}
	The paper \cite{PrRiSe} contains similar Hardy inequalities with the $p$-th power of the sub-Riemannian gradient for any $1<p<\infty$. Our proof extends to this case and, in particular, shows that these inequalities remain valid without regularity assumptions on the boundary.
\end{remark}

\begin{proof}
	We have
	$$
	\sum_{n=1}^{2N} |(Z_n u)(p)|^2 = 2N |\Sph^{2N-1}|^{-1} \int_{\Sph^{2N-1}} |\omega\cdot (Zu)(p)|^2 \,d\omega
	$$
	and consequently
	$$
	\int_\Omega \sum_{n=1}^{2N} |(Z_n u)(p)|^2\,dp = 2N |\Sph^{2N-1}|^{-1} \int_{\Sph^{2N-1}} \left( \int_\Omega |\omega\cdot (Zu)(p)|^2 \,dp \right) d\omega \,.
	$$
	We will bound the inner integral from below for any $\omega\in\Sph^{2N-1}$.
	
	So let $\omega\in\Sph^{2N-1}$ be fixed. Note that any $z\in\R^{2N}$ can be written in a unique fashion as $z=q+s\omega$ with $q\in\omega^\bot=\{ q'\in\R^{2N}:\ \omega\cdot q'=0\}$ and $s\in\R$. Moreover, the Lebesgue measure $dz$ on $\R^{2N}$ is the product measure of the induced Lebesgue measure $d\sigma(q)$ on $\omega^\bot$ and Lebesgue measure $ds$ on $\R$. By an additional linear shift in the $t$ variable, we find that, for any $f\in L^1(\Hei^N)$,
	$$
	\int_{\Hei^N} f(p)\,dp = \int_{\omega^\bot} \int_\R \int_\R f(q+s\omega,t+\frac s2 q\cdot J\omega) \,dt\,ds\,d\sigma(q) \,.
	$$
	
	In our situation, it follows that
	$$
	\int_\Omega |\omega\cdot (Zu)(p)|^2 \,dp = \int_{\omega^\bot} \int_\R \left( \int_\R |\omega\cdot(Zu)(q+s\omega,t+\frac s2 q\cdot J\omega)|^2 \,ds\right) dt\,d\sigma(q) \,.
	$$
	Here we have interchanged the $s$ and $t$ integrals and we have extended $u$ by zero to a function on all of $\Hei^N$. We will bound the inner-most integral from below for any $q\in\omega^\bot$ and $t\in\R$.
	
	So let $\omega\in\Sph^{2N-1}$, $q\in\omega^\bot$ and $t\in\R$ be fixed and consider the curve
	$$
	\gamma(s) := q\circ(s\omega,0) = (q+s\omega,t+\frac s2 q\cdot J\omega) 
	\qquad\text{for}\ s\in\R \,.
	$$
	A short computation shows that
	$$
	\omega\cdot(Zu)(q+s\omega,t+\frac s2 q\cdot J\omega) = (u\circ\gamma)'(s) \,.
	$$
	(Here on the right, of course, the notation $\circ$ means the composition of functions.) Thus, the one-dimensional Hardy inequality implies that
	$$
	\int_\R |\omega\cdot(Zu)(q+s\omega,t+\frac s2 q\cdot J\omega)|^2 \,ds
	= \int_\R |(u\circ\gamma)'(s)|^2\,ds \geq \frac14 \int_\R \frac{|(u\circ\gamma)(s)|^2}{\dist(s,\R\setminus\Omega_{\omega,q,t})^2}\,ds
	$$
	with
	$$
	\Omega_{\omega,q,t} := \{ s'\in\R:\ (q+s'\omega,t+\frac {s'} 2 q\cdot J\omega) \in\Omega \} \,.
	$$
	We observe that
	\begin{align*}
		\dist(s,\R\setminus\Omega_{\omega,q,t}) & = \inf\{ |s-s'| :\ (q+s'\omega,t+\frac {s'} 2 q\cdot J\omega) \not\in\Omega \} \\
		& = \inf\{ |\sigma| :\ (q+(s+\sigma)\omega,t+\frac {s+\sigma} 2 q\cdot J\omega) \not\in\Omega \} \\
		& = \inf\{ |\sigma| :\ (q+s\omega,t+\frac {s} 2 q\cdot J\omega)\circ (\sigma\omega,0) \not\in\Omega \} \\
		& = \tilde\delta_\omega(q+s\omega,t+\frac {s} 2 q\cdot J\omega) \,.
	\end{align*}
	Thus, we have shown that
	$$
	\int_\R |\omega\cdot(Zu)(q+s\omega,t+\frac s2 q\cdot J\omega)|^2 \,ds
	\geq \frac14 \int_\R \frac{|u(q+s\omega,t+\frac s2 q\cdot J\omega)|^2}{\tilde\delta_\omega(q+s\omega,t+\frac {s} 2 q\cdot J\omega)^2}\,ds \,.
	$$
	This inequality holds for all $\omega$, $q$ and $t$. Integrating it with respect to $t$ and $q$ and using the above change of variables formula, we find
	$$
	\int_\Omega |\omega\cdot (Zu)(p)|^2 \,dp \geq \frac14 \int_\Omega \frac{|u(p)|^2}{\tilde\delta_\omega(p)^2}\,dp \,.
	$$
	This is the desired lower bound for fixed $\omega\in\Sph^{2N-1}$. Inserting it in the formula at the beginning of the proof, gives the asserted inequality in the proposition.	
\end{proof}

\begin{proof}[Proof of Theorem \ref{mainheisen}]
	In view of Proposition \ref{hardyheisenberg}, we can proceed as in the proofs of Theorems \ref{main} and \ref{mainPolyharmonic} and so it suffices to prove the analogue of Lemma \ref{lemma} in the Heisenberg setting. More precisely, we shall set
	$$
	\tilde\psi_r(p) := \frac{\sigma(\Omega\cap B_r^{\rm H}(p))}{\sigma(B_r^{\rm H}(p))}
	\qquad\text{for all}\ p\in\Omega \,,\ r>0 \,,
	$$
	and prove for each $p\in\Omega$ and $r>0$
	\begin{equation}
		\label{eq:lemmaheisen1}
		 |\Sph^{2N-1}|^{-1} \int_{\Sph^{2N-1}} \tilde\delta_\omega(p)^{-2}\,d\omega \geq \frac1{r^2} \left( (\tilde\psi_r(p))^{-1/N}- 1 \right)
	\end{equation}
	and
	\begin{equation}
		\label{eq:lemmaheisen2}
		|\Sph^{2N-1}|^{-1} \int_{\Sph^{2N-1}} \tilde\delta_\omega(p)^{-2}\,d\omega \geq \frac{(N+1)^{(N+1)/N}}{N} \ \frac1{r^2} \left( 1-\tilde\psi_r(p) \right).
	\end{equation}
	
	The proof is similar to that of Lemma \ref{lemma}, but we indicate how to deal with some differences. Parametrizing points in the hyperplane $H(p)$ by $p\circ(z',0)$, we see that $d\sigma$ becomes $c dz'$ for a constant $c>0$. (In fact, $c=\sqrt{1+\frac14 |z|^2}$ where $p=(z,t)$, but the value of this normalization constant plays no role, since we look at a quotient of $\sigma$-measures.)
	
	Thus, by introducing spherical coordinates in $\R^{2N}$,
	\begin{align*}
		\sigma(\Omega\cap B^{\rm H}_r(p)) & = c \int_{\Sph^{2N-1}} \int_0^r \1_\Omega(p\circ(\rho\omega,0)) \, \rho^{2N-1}\,d\rho\,d\omega \\
		& \geq c \int_{\Sph^{2N-1}} \int_0^{\min\{r,\tilde\delta_\omega(p)\}} \rho^{2N-1}\,d\rho\,d\omega \\
		& = c (2N)^{-1} \int_{\Sph^{2N-1}} \min\{r^{2N}, \tilde\delta_\omega(p)^{2N} \} \,d\omega \,.
	\end{align*}
	Here we used the fact that $p\circ(\rho\omega,0)\in\Omega$ for all $\rho<\tilde\delta_\omega(p)$. Similarly, we find
	$$
	\sigma(B^{\rm H}_r(p)) = c (2N)^{-1} \int_{\Sph^{2N-1}} r^{2N} \,d\omega \,,
	$$
	so that
	\begin{equation}
		\label{eq:lemmaheisenproof}
		1- \tilde\psi_r(p) = 
		1 - \frac{\sigma(\Omega\cap B^{\rm H}_r(p))}{\sigma(B^{\rm H}_r(p))} \leq |\Sph^{2N-1}|^{-1} \int_{\Sph^{2N-1}} \left( 1 - \left( \tilde\delta_\omega(p)/r \right)^{2N} \right)_+ d\omega \,.
	\end{equation}
	Using the elementary inequality \eqref{eq:elementary} with $X=(\tilde\delta_\omega(p)/r)^{2N}$ and $\beta = 1/N$, we arrive at \eqref{eq:lemmaheisen2}.
	
	To prove \eqref{eq:lemmaheisen1}, we write the bound \eqref{eq:lemmaheisenproof} as
	$$
	\tilde\psi_r(p) \geq 1- 2N \int_0^1 \tilde s(1/\sigma) \, \sigma^{2N-1}\,d\sigma = 1- 2N \int_1^\infty \tilde s(t) \, t^{-2N-1}\,dt
	$$
	with
	$$
	\tilde s(t) := |\Sph^{2N-1}|^{-1} |\{ \omega\in\Sph^{2N-1} :\ r/\tilde\delta_\omega(p)\geq t \}| \,.
	$$
	Meanwhile, by the layer cake formula,
	$$
	r^2 |\Sph^{2N-1}|^{-1} \int_{\Sph^{2N-1}} \frac{d\omega}{\tilde\delta_\omega(p)^2} = \int_0^\infty \tilde s(\tau^{1/2})\,d\tau = 2 \int_0^\infty \tilde s(t) \, t\,dt \,.
	$$
	In view of these formulas, the claimed bound \eqref{eq:lemmaheisen1} follows from \eqref{eq:functionineqality} with $\alpha=2$ and $d=2N$.
\end{proof}


\appendix

\section{Further remarks on Robin eigenvalues}

The following results complement those of Kovařík in \cite{Kov14}.

\begin{proposition}\label{appendix1}
	Let $\Omega\subset\R^d$ be an open set with $C^2$ boundary and with nonnegative mean curvature. Then for all $u\in H^1(\Omega)$ and all $\sigma>0$,
	$$
	\int_\Omega |\nabla u(x)|^2 \,dx + \sigma \int_{\partial\Omega} |u(y)|^2\,dS(y) \geq \frac14 \frac{\sigma}{R_\Omega\,(1+\sigma R_\Omega)} \int_\Omega |u(x)|^2\,dx \,.
	$$
\end{proposition}

\begin{proof}
	We follow the proof of \cite[Lemma 4.3]{Kov14} and obtain, for each $\alpha>0$,
	$$
	\int_\Omega |\nabla u(x)|^2 \,dx + \sigma \int_{\partial\Omega} |u(x)|^2\,dS(x) \geq \alpha\sigma(1-\alpha\sigma) \int_\Omega \frac{|u(x)|^2}{(\delta(x)+\alpha)^2}\,dx
	$$
	where $\delta(x):=\dist(x,\R^d\setminus\Omega)$. Indeed, the proof of this inequality only uses the fact that $\Delta\delta\leq 0$ in the sense of distributions, which is valid under our assumptions on $\Omega$; see \cite{LeLiLi} and the references therein.
	
	Now we proceed as in \cite[Theorem 4.4]{Kov14}, bounding $\delta\leq R_\Omega$ and choosing $\alpha = R_\Omega/(1+2\sigma R_\Omega)$. This gives the claimed inequality.
\end{proof}

\begin{proposition}\label{appendix2}
	Let $\Omega\subset\R^d$ be a convex open set. Then for all $u\in H^1(\Omega)$ and all $\sigma>0$,
	$$
	\int_\Omega |\nabla u(x)|^2 \,dx + \sigma \int_{\partial\Omega} |u(y)|^2\,dS(y) \geq \frac{2\sigma^2}{(1+2 \sigma R_\Omega)^2} \int_\Omega |u(x)|^2\,dx \,.
	$$
\end{proposition}

When $\sigma R_\Omega$ is small, this bound is inferior to the bound in Proposition \ref{appendix1}, but when it is large, the bound in Proposition \ref{appendix2} is superior by about a factor of 2.

\begin{proof}
	We know from \cite{KL12} that
	\begin{align*}
		\int_{\Omega}|\nabla u(x)|^2 \,d x+ \sigma \int_{\partial \Omega} |u(y)|^2 d S(y)  
		& \geq \frac{1}{4} \int_{\Omega}\left(\delta(x)+\frac{1}{2 \sigma}\right)^{-2}|u(x)|^2 \,d x \nonumber \\
		& \quad +\frac{1}{4} \int_{\Omega}\left(R_{\Omega}+\frac{1}{2 \sigma}\right)^{-2}|u(x)|^2 \,d x  \nonumber  \\
		& \quad +\frac{1}{2} \int_{\partial \Omega}\left(R_{\Omega}+\frac{1}{2 \sigma}\right)^{-1}|u(y)|^2 \,d S(y) \,.
	\end{align*}
	We drop the last term on the right side and bound $\delta\leq R_\Omega$ in the first term. This yields the claimed bound.
\end{proof}


\end{document}